\documentclass[10pt]{amsart}
\usepackage{hyperref}
\usepackage{amsmath}
\usepackage{amssymb}
\usepackage{xypic}
\usepackage{yhmath}
\usepackage{setspace}
\def\beqnn{\begin{eqnarray*}}\def\eeqnn{\end{eqnarray*}}

\newtheorem{theorem}{Theorem}[section]
\newtheorem{lemma}[theorem]{Lemma}
\newtheorem{proposition}[theorem]{Proposition}
\newtheorem{question}[theorem]{Question}
\newtheorem{problem}[theorem]{Problem}

\theoremstyle{definition}

\theoremstyle{remark}
\newtheorem{remark}[theorem]{Remark}

\numberwithin{equation}{section}

\begin{document}

\begin{center}
\title[Prawitz's area theorem and mixed Aharonov sequence]{Prawitz's area theorem and \\ mixed Aharonov sequence}
\end{center}

\author{Jianjun Jin}
\address{School of Mathematics Sciences, Hefei University of Technology, Xuancheng Campus, Xuancheng 242000, P.R.China}
\email{jin@hfut.edu.cn, jinjjhb@163.com}

\begin{abstract}
In this paper, motivated by the Prawitz area theorem and the work of Aharonov, we introduce the mixed Aharonov sequence associated with a locally univalent analytic function. By using the mixed Aharonov sequence, we establish a new univalence criterion for the locally univalent analytic functions in the unit disk, which generalizes some related results of Aharonov in \cite{Ah}. We also prove some new properties about the (mixed) Aharonov sequence, in particular, a new inequality for the Aharonov sequence is established for the univalent functions with a quasiconformal extension. 
\end{abstract}

\maketitle

\section{{\bf {Introduction}}}
Let $\mathbb{D}=\{z:|z|<1\}$ be the unit disk in the complex plane $\mathbb{C}$.  
We denote by $\widehat{\mathbb{C}}=\mathbb{C}\cup\{\infty\}$ the extended complex plane.  Let $\mathbb{D}_{e}=\widehat{\mathbb{C}}-\overline{\mathbb{D}}$ be the
exterior of $\mathbb{D}$ and $\mathbb{T}=\partial\mathbb{D}=\partial\mathbb{D}_e$ be the unit circle.  

We say $f$ is a univalent analytic (meromorphic) function in an open domain $\Omega$ of $\widehat{\mathbb{C}}$, if $f$ is a one-to-one and analytic (meromorphic) function in $\Omega$.  Let $\mathcal{U}(\mathbb{D})$ be the class of all univalent analytic functions $f$ in $\mathbb{D}$. We denote by $S$ the class of all univalent analytic functions $f$ in  $\mathbb{D}$ with $f(0)=f'(0)-1=0$.

We denote by $\Sigma$ the class of all univalent meromorphic functions $g$ in  $\mathbb{D}_e$ with the following expansion
at infinity as
$$g(z)=z+b_0+\frac{b_1}{z}+\frac{b_2}{z^2}+\cdots. $$

It is well known that the area theorems play an important role in the study of the theory of univalent analytic (meromorphic) functions. The first area theorem was proved by Gr\"onwall in 1914, see \cite[page 18]{Po}, which states that
\begin{theorem}\label{g-1}If $g\in \Sigma$, then we have 
\begin{equation}\label{g-ine}
\frac{1}{\pi}\iint_{\mathbb{D}_e}|g'(z)-1|^2dxdy=\sum_{n=1}^{\infty}n|b_n|^2 \leq 1.
\end{equation}
Equality holds in (\ref{g-ine}) if and only if $g$ is a full mapping.
\end{theorem}
\begin{remark}
Here we say $g$ is a full mapping if the Lebesgue measure of $\widehat{\mathbb{C}}\setminus g(\mathbb{D}_e)$ is zero. 
\end{remark}

If $f\in S$, then we see that $g(z)=\frac{1}{f(z^{-1})}$ belongs to $\Sigma$. It follows from Theorem \ref{g-1} that
\begin{theorem}\label{g-2}If $f\in S$, then we have 
\begin{equation}\label{g-ine-1}
\frac{1}{\pi}\iint_{\mathbb{D}}\left|f'(z)\left[\frac{z}{f(z)}\right]^{2}-1\right|^2\frac{dxdy}{|z|^{4}}=\sum_{n=1}^{\infty}n|a_n|^2\leq 1. 
\end{equation}
Where $$\frac{1}{f(z)}=\frac{1}{z}+\sum_{n=0}^{\infty}a_nz^n.$$
Equality holds in (\ref{g-ine-1}) if and only if $f$ is a full mapping.
\end{theorem}
\begin{remark}
Here we say $f$ is a full mapping if the Lebesgue measure of $\mathbb{C}\setminus f(\mathbb{D})$ is zero.
\end{remark}

In 1927, Prawitz proved  in \cite{Pra} a generalized area theorem(see also \cite[page 13]{Mi}).  Let $\lambda>0$. Let $g\in \Sigma$ with $g(z)\neq0$ for all $z\in \mathbb{D}_e$. We see that
\begin{equation}
\left[\frac{g(z)}{z}\right]^{\lambda}=\exp\left[\lambda \ln\frac{g(z)}{z}\right]\nonumber 
\end{equation}
is  analytic in $\mathbb{D}_e$. Here the logarithm expression is determined uniquely by requiring that it is analytic in $\mathbb{D}_e$ and takes the value zero at infinity.  

Then we have the following Prawitz area theorem.
\begin{theorem}\label{pra}
Let $\lambda>0$. Let $g\in \Sigma$ with $g(z)\neq0$ for all $z\in \mathbb{D}_e$.  Then we have 
\begin{equation}\label{p-ine}
\sum_{n=1}^{\infty}(n-\lambda)|b_n(\lambda)|^2 \leq \lambda.
\end{equation}
Where \begin{equation}
\left[\frac{g(z)}{z}\right]^{\lambda}=1+\sum_{n=1}^{\infty}b_n(\lambda) z^{-n}.\nonumber 
\end{equation}
Equality holds in (\ref{p-ine}) if and only if $g$ is a full mapping.
\end{theorem}

If $f\in S$,  then we get that $g(z)=\frac{1}{f(z^{-1})}\in \Sigma$ with $g(z)\neq 0$ for all $z\in \mathbb{D}_e$ and 
$$\frac{z}{f(z)}=\frac{g(z^{-1})}{z^{-1}},\, z\in \mathbb{D}.$$ 

From Theorem \ref{pra},  we obtain the following version of Prawitz area theorem for the class $S$.
\begin{theorem}\label{pra-s}
Let $\lambda>0$. Let $f\in S$.  Then we have 
\begin{equation}\label{ps-ine}
\sum_{n=1}^{\infty}(n-\lambda)|a_n(\lambda)|^2 \leq \lambda.
\end{equation}
Where \begin{equation}
\left[\frac{z}{f(z)}\right]^{\lambda}=1+\sum_{n=1}^{\infty}a_n(\lambda) z^{n}.\nonumber 
\end{equation}
Equality holds in (\ref{ps-ine}) if and only if $f$ is a full mapping.
\end{theorem}

Under the assumption of  Theorem \ref{pra}, when $\lambda\in (0,1)$, we first notice that
\begin{eqnarray}\label{n-1}
\lambda\left[g'(z)\left[\frac{z}{g(z)}\right]^{1-\lambda}-1\right]&=&z\frac{d}{dz}\left[\frac{g(z)}{z}\right]^{\lambda}+\lambda\left[\frac{g(z)}{z}\right]^{\lambda}-\lambda \nonumber \\
&=&-\sum_{n=1}^{\infty}(n-\lambda)b_n(\lambda) z^{-n}.
\end{eqnarray}
On the other hand, we have
\begin{eqnarray}\label{n-2}
\lefteqn{\frac{1}{\pi}\iint_{\Delta^*}\left|\sum_{n=1}^{\infty}(n-\lambda)b_n(\lambda)  z^{-n}\right|^2 \frac{1}{|z|^{2(1-\lambda)}} dxdy} \nonumber \\
&&=\frac{1}{\pi} \iint_{\Delta^*}\sum_{n=1}^{\infty}(n-\lambda)^2|b_n(\lambda)|^2 |z|^{-2n+2\lambda-2} dxdy \nonumber \\
&&=2\sum_{n=1}^{\infty}(n-\lambda)^2|b_n(\lambda)|^2\int_{1}^{\infty}r^{-2n+2\lambda-2}rdr \nonumber \\
&&=\sum_{n=1}^{\infty}(n-\lambda)|b_n(\lambda)|^2.
\end{eqnarray}
Therefore, we see from (\ref{n-1}) and (\ref{n-2}) that Prawitz's inequality (\ref{p-ine}) can be rewritten as
\begin{equation}\label{p-ine-e}
\frac{1}{\pi}\iint_{\Delta^*}\left|g'(z)\left[\frac{z}{g(z)}\right]^{1-\lambda}-1\right|^2\frac{dxdy}{|z|^{2(1-\lambda)}}\leq \frac{1}{\lambda},
\end{equation} for $\lambda\in (0,1]$. Note that (\ref{p-ine-e}) becomes (\ref{g-ine}) when $\lambda=1$.
Replacing $g$ in (\ref{p-ine-e}) by $\frac{1}{f(z^{-1})}$ for $f\in S$, we obtain that
\begin{equation}
\frac{1}{\pi}\iint_{\Delta^*}\left|\frac{f'(z^{-1})}{[zf(z^{-1}]^2}\left[zf(z^{-1})\right]^{1-\lambda}-1\right|^2\frac{dxdy}{|z|^{2(1-\lambda)}}\leq \frac{1}{\lambda},\nonumber 
\end{equation}
for $\lambda\in (0,1]$.  
By the change of variables $z=\frac{1}{w}$, we get that
\begin{equation}\label{p-ine-d}
\frac{1}{\pi}\iint_{\mathbb{D}}\left|f'(w)\left[\frac{w}{f(w)}\right]^{1+\lambda}-1\right|^2\frac{dudv}{|w|^{2(1+\lambda)}}\leq \frac{1}{\lambda},
\end{equation}
for $\lambda\in (0,1]$. 

For $z\in \mathbb{D}$,  we set 
\begin{equation}\label{sigma}\sigma_{z}(\zeta):=\frac{\zeta+z}{1+\bar{z}\zeta}, \, \zeta \in \mathbb{D}.\end{equation}

Now, let $f\in \mathcal{U}(\mathbb{D})$,  the Koebe transformation $\mathbf{K}_f(z; \zeta)$ of $f$ is defined as 
\begin{equation}\label{Koe}\mathbf{K}_f(z; \zeta):=\frac{f(\sigma_{z}(\zeta))-f(z)}{(1-|z|^2)f'(z)}=\frac{f\left(\frac{\zeta+z}{1+\bar{z}\zeta}\right)-f(z)}{(1-|z|^2)f'(z)},\, \zeta\in \mathbb{D}.\end{equation}
We see that $\mathbf{K}_f(z; \zeta)$ belongs to $S$.
By using $\mathbf{K}_f(z; \zeta)$ to be instead of $f$ in  (\ref{p-ine-d}), we get that 
\begin{equation}
\frac{1}{\pi}\iint_{\mathbb{D}}\left|\frac{ f'(\sigma_{z}(\zeta))}{f'(z)(1+\bar{z}\zeta)^2}
\left[\frac{(1-|z|^2)f'(z)\zeta}{f(\sigma_{z}(\zeta))-f(z)}\right]^{1+\lambda}-1\right|^2\frac{d\xi d\eta}{|\zeta|^{2(1+\lambda)}}\leq \frac{1}{\lambda}. \nonumber
\end{equation}

Then, after the change of variables $\zeta=\frac{w-z}{1-\bar{z}w}$,  we obtain that, for $\lambda\in (0, 1]$,
\begin{equation}\label{main-ine}
\frac{(1-|z|^2)^{2\lambda}}{\pi}\iint_{\mathbb{D}}\left|P(f; z, w)\right|^2\frac{dudv}{|w-z|^{2(1+\lambda)}} \leq  \frac{1}{\lambda}.
\end{equation}
Here 
\begin{equation}
P(f; z, w)=\frac{f'(w)(w-z)}{f(w)-f(z)}
\left[\frac{f'(z)(w-z)}{f(w)-f(z)}\right]^{\lambda}-\left(\frac{1-|z|^2}{1-\bar{z}w}\right)^{1-\lambda},\, z, w \in\mathbb{D}.\nonumber
\end{equation}

\begin{remark}
From the above arguments, we see that Prawitz’s inequality (\ref{main-ine}) holds for all $f\in \mathcal{U}(\mathbb{D})$.  The inequaity (\ref{main-ine}) has appeared in \cite{HS-1},  where (\ref{main-ine})  was used to study the integral means spectrum of univalent analytic functions in $\mathbb{D}$.
\end{remark}
\begin{remark} When $\lambda=1$,  (\ref{main-ine}) reduces to 
\begin{equation}
\frac{(1-|z|^2)^{2}}{\pi}\iint_{\mathbb{D}}|U(f; z,w)|^2dudv\leq 1. \nonumber
\end{equation}
Here $$U(f; z,w)=\frac{f'(z)f'(w)}{[f(w)-f(z)]^2}-\frac{1}{(z-w)^2},\,\, z, w\in \mathbb{D},$$
is known as Grunsky kernel. We set
$$U_f(z):=\left(
\frac{1}{\pi}\iint_{\mathbb{D}}|U(f; z,w)|^2dudv\right)^{\frac{1}{2}},\, z\in \mathbb{D}.$$
It is known that $U(f; z,w)$ and $U_f(z)$ also play an important role in the study of univalent analytic (meromorphic) functions and the universal Teichm\"uller space, see for example \cite{Ba}, \cite{Ha}, \cite{Sh-1}, \cite{Sh-2}, \cite{TT}.\end{remark}

In this note, motivated by the Prawitz area theorem and the work of Aharonov \cite{Ah}, we introduce the mixed Aharonov sequence associated with a locally univalent analytic function, and use it to establish a new univalence criterion for the locally univalent analytic functions in the unit disk.  The univalence criterion will be given in Section 2 and the proof of our first main results will be presented in Section 3.  In Section 4, we will remark some new results on the (mixed) Aharonov sequence and raise some questions, in particular, we obtain a new inequality (see below (\ref{jia})) for the Aharonov sequence for the univalent functions with a quasiconformal extension.
 
\section{\bf {The mixed Aharonov sequence and a generalized univalence criterion for locally univalent functions}}

Astala and Gehring have studied in \cite{AG} the following 
\begin{problem}
For a locally univalent analytic (or meromorphic) function $f$ in the unit disk $\mathbb{D}$,  what additional conditions on $f$ allow one to conclude that $f$ is univalent?
\end{problem}

In this note we continue to consider this problem.  Motivated by the classical work of Aharonov \cite{Ah} and by investigating the Prawitz area theorem, we introduce the notion of {\em mixed Aharonov sequence}  and use it to establish a new univalence criterion for the locally univalent analytic functions in the unit disk. We first recall Aharonov's results proved in \cite{Ah}. 

We denote by $\mathcal{M}(\mathbb{D})$ the class of all locally univalent meromorphic functions in $\mathbb{D}$. We note that any pole of the function $f\in \mathcal{M}(\mathbb{D})$  is of the first order. 

We denote by $\mathcal{L}(\mathbb{D})$ the class of all locally univalent analytic functions $f$ in $\mathbb{D}$, i.e., $f$ is analytic in $\mathbb{D}$ with $f'(z)\neq 0$ for all $z\in \mathbb{D}$. It is easy to see that $\mathcal{L}(\mathbb{D})$ is contained in $\mathcal{M}(\mathbb{D})$.

Let $f\in \mathcal{M}(\mathbb{D})$. For $z\in \mathbb{D}$ with it is not the pole of $f$, we take a point $w$, which is close enough to $z$, and define the sequence $\{\phi_n(f; z)\}$ by the expansion of the following generating function,
\begin{equation}\label{inv}
\frac{f'(z)}{f(w)-f(z)}=\frac{1}{w-z}+\sum_{n=0}^{\infty}\phi_n(f; z)(w-z)^n.
\end{equation}
The sequence $\{\phi_n(f; z)\}$ is called as {\em Aharonov sequence}, which has some fine properties. 

We have
\begin{equation} 
\phi_0(f; z)=-\frac{1}{2}N_f(z)=-\frac{1}{2}\frac{f''(z)}{f'(z)};\nonumber
\end{equation}
\begin{equation}\label{inv-1}
\phi_1(f; z)=-\frac{1}{6}S_f(z)=-\frac{1}{6}\left[\left(\frac{f''(z)}{f'(z)}\right)'-\frac{1}{2}\left(\frac{f''(z)}{f'(z)}\right)^2\right];
\end{equation}
\begin{equation}\label{inv-2}
\phi_{n+1}(f; z)=\frac{1}{n+3}\left[\phi'_{n}(f; z)-\sum_{k=1}^{n-1}\phi_{k}(f; z)\phi_{n-k}(f; z)\right],\, n\geq 1.
\end{equation}
Here $S_f$ is the Schwarzian derivative of $f$, and $N_f$ is called Pre-Schwarzian derivative of $f$. It is well known that the Schwarzian derivative $S_f$ is M\"obius invariant, see \cite[Chapter II]{L}.  It follows from (\ref{inv-1}) and (\ref{inv-2}) that any $\phi_n(f; z)(n\geq 1)$ is  M\"obius invariant in the sense that 
$$\phi_n(\tau \circ f; z)= \phi_n(f; z)$$
for all M\"obius transformation $\tau$. So $\phi_n(f; z)(n\geq 1)$ are often called as Aharonov invariants.

\begin{remark}
When $f\in \mathcal{M}(\mathbb{D})$ has a first order pole $z$,  we can define  $\phi_n(f; z)$ in view of  (\ref{inv-1}) and (\ref{inv-2}) for all $n \geq 1$.
\end{remark}

From \cite{Ah}, we have 
\begin{theorem}\label{ah-1}
Let $f\in \mathcal{M}(\mathbb{D})$. Then $f$ is univalent meromorphic in $\mathbb{D}$ if and only if, for all $\zeta\in \mathbb{D}$, it holds that
\begin{equation}\label{a-1}
\sum_{n=1}^{\infty}n\left|\sum_{k=1}^n \binom{n-1}{k-1} (-\bar{\zeta})^{n-k}(1-|\zeta|^2)^{k+1}\phi_k(f; \zeta)\right|^2\leq 1.
\end{equation}
\end{theorem}

\begin{theorem}\label{ah-2}
Let $f\in \mathcal{M}(\mathbb{D})$. Then $f$ is univalent meromorphic and a full mapping in $\mathbb{D}$, if and only if,  for all $\zeta\in \mathbb{D}$, it holds that
\begin{equation}
\sum_{n=1}^{\infty}n\left|\sum_{k=1}^n \binom{n-1}{k-1} (-\bar{\zeta})^{n-k}(1-|\zeta|^2)^{k+1}\phi_k(f; \zeta)\right|^2 \equiv 1. \nonumber
\end{equation}
\end{theorem}

\begin{theorem}\label{ah-3}
Let $f\in \mathcal{M}(\mathbb{D})$. Then $f$ is univalent meromorphic in $\mathbb{D}$ if and only if, for all $\zeta\in \mathbb{D}$, there is a positive constant $C(\zeta)$ such that for all $n\geq 1$,
\begin{equation}\left|\sum_{k=1}^n \binom{n-1}{k-1} (-\bar{\zeta})^{n-k}(1-|\zeta|^2)^{k+1}\phi_k(f; \zeta)\right|\leq C(\zeta).\nonumber
\end{equation}
\end{theorem}

\begin{remark}In the above theorems, the combinatorial number $\binom{\alpha}{m}$ is defined as 
$$\binom{\alpha}{m}:=\frac{(\alpha)_m}{m!}=\frac{\alpha(\alpha-1)\cdots (\alpha-m+1)}{m!}.$$
Where $\alpha\in \mathbb{R}$ and $m$ is a non-negative integer number. In particular, we have
 $$\binom{\alpha}{0}=1\,\,  {\text{and}}\,\, \binom{-\alpha}{m}=(-1)^m\binom{\alpha+m-1}{m},\, {\textup {for any}}\, \alpha\in \mathbb{R}.$$
It should be pointed out that, in the recent paper \cite{Su}, Sugawa has reformulated above univalence criterion by using projective Schwarzian derivatives of higher order. 
 \end{remark}

Let $f\in \mathcal{L}(\mathbb{D})$. For $\lambda>0$,  we next consider the following generating function and its expansion,
\begin{equation}\label{new}
\left[\frac{f'(z)(w-z)}{f(w)-f(z)}\right]^{\lambda}=\sum_{n=0}^{\infty}\Phi_{\lambda, n}(f; z)(w-z)^n.
\end{equation}
 For the sake of simplicity, we will occasionally use  $\Phi_n(f; z)$ or $\Phi_n$ to denote $\Phi_{\lambda, n}(f; z)$, and use $\phi_n$ to denote $\phi_{n}(f; z)$.   
 
From (\ref{inv}), we get that 
 \begin{eqnarray}
\left[\frac{f'(z)(w-z)}{f(w)-f(z)}\right]^{\lambda}&=&\left[1+\sum_{n=0}^{\infty}\phi_n(w-z)^{n+1}\right]^{\lambda} \nonumber \\
&=&\sum_{n=0}^{\infty}\binom{\lambda}{n} \left[\sum_{k=0}^{\infty}\phi_k(w-z)^{k+1}\right]^{n}.\nonumber 
\end{eqnarray} 
Then we see that $\Phi_0\equiv 1$ and 
\begin{eqnarray}
1+\sum_{n=1}^{\infty}\Phi_{n}(w-z)^n=1+\sum_{n=1}^{\infty}\binom{\lambda}{n}(w-z)^n \left[\sum_{k=0}^{\infty}\phi_k(w-z)^{k}\right]^{n}.\nonumber \end{eqnarray}

Note that 
\begin{equation}
 \left[\sum_{k=0}^{\infty}\phi_n(w-z)^{k}\right]^{n}=\sum_{k_1, k_2, \cdots, k_n=0}^{\infty}\phi_{k_1}\phi_{k_2}\cdots\phi_{k_n}(w-z)^{k_1+k_2+\cdots+k_n},\, n\geq 1. \nonumber
\end{equation}
It follows that 
\begin{eqnarray}\label{Phi}
\Phi_n=\sum_{j=1}^{n}\binom{\lambda}{j} \sum_{(k_1, k_2, \cdots, k_j)\in I(n,j)}\phi_{k_1}\phi_{k_2}\cdots\phi_{k_j},\, n\geq 1. \end{eqnarray}
Here, $I(n,j)$ is the set of all $j$-tuples $(k_1,\cdots, k_j)$ of non-negative integers with $k_1+\cdots+k_j=n-j$. 

In particular,  we have $$\Phi_1(f; z)=\lambda \phi_0(f; z)=-\frac{\lambda}{2}N_f(z);$$
$$\Phi_2(f; z)=\lambda \phi_1(f; z)+\frac{\lambda(\lambda-1)}{2}\phi_0^2(f; z)=-\frac{\lambda}{6}S_f(z)+\frac{\lambda(\lambda-1)}{8}N_f^2(z).$$

\begin{remark}We get from (\ref{Phi}) that $\Phi_n(n\geq 0)$ is affine invariant in the sense that 
$$\Phi_n(\chi \circ f; z)= \Phi_n(f; z)$$
for all affine transformation $\chi(z)=az+b$ since $\phi_0$ is  affine invariant and $\phi_n(n\geq 1)$ is M\"obius invariant. 
From this fact, we see that we can't define $\Phi_n(f; z)$ if $f$ belongs to $\mathcal{M}(\mathbb{D})$ and $z$ is a pole of $f$. We will call $\{\Phi_n(f; z)\}$ as {\em  mixed Aharonov sequence}. (\ref{Phi}) explicitly shows how the mixed sequence is determined by the classical Aharonov invariants.
\end{remark}

We now state our first main results of this paper. We denote
$$A_n(f; \zeta):=\sum_{j=0}^{n}(-1)^{n-j}\binom{\lambda}{n-j}\sum_{k=0}^{j}\binom{j-1}{j-k}(-\bar{\zeta})^{n-k}(1-|\zeta|^2)^{k}\Phi_{k}(f; \zeta).$$

\begin{theorem}\label{main-1}
Let $\lambda>0$ and $f \in \mathcal{L}(\mathbb{D})$. Then $f$ is univalent analytic in $\mathbb{D}$ if and only if, for all $\zeta\in \mathbb{D}$, it holds that
\begin{equation}\label{m-ine-1}
\sum_{n=1}^{\infty}(n-\lambda)\left|A_n(f; \zeta)\right|^2\leq \lambda.
\end{equation}
\end{theorem}

\begin{theorem}\label{main-2}
Let $\lambda>0$ and $f \in \mathcal{L}(\mathbb{D})$. Then $f$ is univalent analytic and a full mapping in $\mathbb{D}$, if and only if,  for all $\zeta\in \mathbb{D}$, it holds that
\begin{equation}\label{m-ine-2}
\sum_{n=1}^{\infty}(n-\lambda)\left|A_n(f; \zeta)\right|^2\equiv \lambda.
\end{equation}
\end{theorem}

\begin{theorem}\label{main-3}
Let $\lambda>0$ and $f \in \mathcal{L}(\mathbb{D})$. Then $f$ is univalent analytic in $\mathbb{D}$ if and only if, for all $\zeta\in \mathbb{D}$, there is a positive constant $C(\lambda, \zeta)$ such that for all $n\geq 1$,
\begin{equation}\label{m-ine-3}\left|A_n(f; \zeta)\right|\leq C(\lambda, \zeta).
\end{equation}
\end{theorem}

\begin{remark}
Note that (\ref{m-ine-1}) can be equivalently written as
\begin{equation} 
\sum_{n>\lambda}(n-\lambda)\left|A_n(f; \zeta)\right|^2\leq \lambda+\sum_{n\leq \lambda}(\lambda-n)\left|A_n(f; \zeta)\right|^2. \nonumber
\end{equation}
This formulation has the advantage that both sides consist of non‑negative terms, making the role of low‑order coefficients explicit.
\end{remark}
\begin{remark}
Due to the arbitrariness of $\lambda>0$, our new criteria are more flexible than the classical Aharonov criterion (which corresponds to the special case $\lambda=1$. The unified formulation also reveals a previously unnoticed connection between the Prawitz area theorem and Aharonov’s univalence criterion. It is hoped that the mixed Aharonov sequence and these criteria will find applications in future research.
\end{remark}

\section{\bf Proof of Theorem \ref{main-1},  \ref{main-2} and \ref{main-3}}
We need the following lemmas in our proof of Theorem \ref{main-1},  \ref{main-2} and \ref{main-3}.
\begin{lemma}\label{lemma-1}
Let  $\lambda>0$ and $f\in \mathcal{U}(\mathbb{D})$. Then we have
 \begin{equation}\label{p-area}
 \sum_{n=1}^{\infty}(n-\lambda)|\Phi_n(f; 0)|^2\leq \lambda.
 \end{equation}
 Equality holds in (\ref{p-area}) if and only if $f$ is a full mapping.
\end{lemma}
\begin{proof}
We may assume that $f\in S$ since $\Phi_n(n\geq 0)$ is affine invariant.  Noticing that
\begin{equation}\label{l-1}
\left[\frac{z}{f(z)}\right]^{\lambda}=\left[\frac{f'(0)(z-0)}{f(z)-f(0)}\right]^{\lambda}=1+\sum_{n=1}^{\infty}\Phi_{n}(f; 0)z^n, \, z\in \mathbb{D}.
\end{equation}  
Then the lemma follows by using Theorem \ref{pra-s}.  The proof of Lemma \ref{lemma-1} is complete.
\end{proof}

\begin{lemma}\label{lemma-2}
Let $\lambda>0$ and $f\in \mathcal{L}(\mathbb{D})$. Let $w, \zeta\in \mathbb{D}$ be fixed and let $z=\sigma_{\zeta}(w)=\frac{w+\zeta}{1+\bar{\zeta}w}$.  We set
$$\mathbf{F}(w)=f(z).$$
Then we have 
\begin{eqnarray}\label{l-22}
\lefteqn{\Phi_{n}(\mathbf{F}; w)=} \nonumber\\
&&\sum_{j=0}^{n}(-1)^{n-j}\binom{\lambda}{n-j}\sum_{k=0}^{j}\binom{j-1}{j-k}\frac{(-\bar{\zeta})^{n-k}(1-|\zeta|^2)^{k}}{(1+\bar{\zeta}w)^{n+k}}\Phi_{k}(f; z), \, n\geq 0.
\end{eqnarray}
\end{lemma} 
\begin{proof}
Choose $\Delta z$ with $|\Delta z|$ small enough, we define $\Delta w$ by 
\begin{equation}
z+\Delta z=\frac{(w+\Delta w)+\zeta}{1+\bar{\zeta}(w+\Delta w)}.\nonumber 
\end{equation}
Then we get that 
\begin{equation}\label{d-con}
\Delta z=\frac{(1-|\zeta|^2)\Delta w}{[1+\bar{\zeta}(w+\Delta w)](1+\bar{\zeta} w)}=\frac{1-|\zeta|^2}{-\bar{\zeta}(1+\bar{\zeta}w)} \frac{d}{1-d}, \, d=\frac{-\bar{\zeta}}{1+\bar{\zeta} w} \Delta w,
\end{equation}
and
\begin{equation}
\frac{\Delta w}{\Delta z}=\left(1+\frac{\bar{\zeta} \Delta w }{1+\bar{\zeta}w}\right) \frac{(1+\bar{\zeta}w)^2}{1-|\zeta|^2}.\nonumber 
\end{equation}
It follows that 
\begin{eqnarray}
\frac{\mathbf{F}'(w)}{\mathbf{F}(w+\Delta w)-\mathbf{F}(w)}=\frac {f'(z)}{f(z+\Delta z)-f(z)} \frac{1-|\zeta|^2}{(1+\bar{\zeta}w)^2}, \nonumber 
\end{eqnarray}
and 
\begin{eqnarray}
\left[\frac{\mathbf{F}'(w)\Delta w}{\mathbf{F}(w+\Delta w)-\mathbf{F}(w)}\right]^{\lambda}&=&\left[\frac{f'(z)\Delta z}{f(z+\Delta z)-f(z)}\frac{\Delta w}{\Delta z} \frac{1-|\zeta|^2}{(1+\bar{\zeta}w)^2}\right]^{\lambda} \nonumber \\
&=& \left[\frac{f'(z)\Delta z}{f(z+\Delta z)-f(z)}\right]^{\lambda}\left[1+\frac{\bar{\zeta} \Delta w }{1+\bar{\zeta}w}\right]^{\lambda}.\nonumber 
\end{eqnarray}
Thus we obtain that 
\begin{eqnarray}\label{eee}
\sum_{n=0}^{\infty}\Phi_{n}(\mathbf{F}; w)(\Delta w)^n=\left[\sum_{n=0}^{\infty}\Phi_{n}(f; z)(\Delta z)^n\right]\left[1+\frac{\bar{\zeta} \Delta w }{1+\bar{\zeta}w}\right]^{\lambda}.
\end{eqnarray}

On the other hand, from (\ref{d-con}), we have 
\begin{eqnarray}\label{con-1}
\lefteqn{\sum_{n=0}^{\infty}\Phi_{n}(f; z)(\Delta z)^n=\sum_{n=0}^{\infty}\Phi_{n}(f; z)\frac{(1-|\zeta|^2)^n}{(-\bar{\zeta})^n(1+\bar{\zeta}w)^n}\cdot d^n {(1-d)^{-n}}}\nonumber \\
&&= \sum_{n=0}^{\infty}\Phi_{n}(f; z) \frac{(1-|\zeta|^2)^n}{(-\bar{\zeta})^n(1+\bar{\zeta}w)^n}\cdot d^n \sum_{m=0}^{\infty}\binom{-n}{m}{(-d)}^{m} \nonumber  \\
&&= \sum_{n=0}^{\infty}\Phi_{n}(f; z) \frac{(1-|\zeta|^2)^n}{(-\bar{\zeta})^n(1+\bar{\zeta}w)^n}\cdot d^n \sum_{m=0}^{\infty}\binom{n+m-1}{m}{d}^{m}
\end{eqnarray}
Then we get  that the coefficient of $(\Delta w)^j$ in (\ref{con-1}) is 
\begin{equation}
\sum_{k=0}^{j}\binom{j-1}{j-k}\frac{(-\bar{\zeta})^{j-k}(1-|\zeta|^2)^{k}}{(1+\bar{\zeta}w)^{j+k}}\Phi_{k}(f; z), \, j\geq 0. \nonumber
\end{equation}

Note that 
$$\left[1+\frac{\bar{\zeta} \Delta w }{1+\bar{\zeta}w}\right]^{\lambda}=\sum_{n=0}^{\infty}\binom{\lambda}{n} \frac{{\bar{\zeta}}^n }{(1+\bar{\zeta}w)^n}(\Delta w)^n.$$
It follows that  the coefficient of $(\Delta w)^n$ in the right side of (\ref{eee}) is 
\begin{equation}
\sum_{j=0}^{n}(-1)^{n-j}\binom{\lambda}{n-j}\sum_{k=0}^{j}\binom{j-1}{j-k}\frac{(-\bar{\zeta})^{n-k}(1-|\zeta|^2)^{k}}{(1+\bar{\zeta}w)^{n+k}}\Phi_{k}(f; z), \, n\geq 0.
\end{equation}

Thus, from (\ref{eee}),  we find that (\ref{l-22}) holds. The lemma is proved. 
\end{proof}

In the rest of this section, we present the proof of Theorem \ref{main-1},  \ref{main-2} and \ref{main-3}.  We will use the same notation as in Lemma \ref{lemma-2}.
\begin{proof}[Proof of Theorem \ref{main-1} and \ref{main-3}]
We first prove the only if part. We assume that $f\in \mathcal{U}(\mathbb{D})$.  Let $w, \zeta\in \mathbb{D}$ be fixed and let $z=\frac{w+\zeta}{1+\bar{\zeta}w}$.  We set $\mathbf{F}(w)=f(z)$.  Then we have 
\begin{eqnarray}\label{pro-1}
\lefteqn{\Phi_{n}(\mathbf{F}; 0)=} \\
&&\sum_{j=0}^{n}(-1)^{n-j}\binom{\lambda}{n-j}\sum_{k=0}^{j}\binom{j-1}{j-k}(-\bar{\zeta})^{n-k}(1-|\zeta|^2)^{k}\Phi_{k}(f; \zeta),\, n \geq 0, \nonumber
\end{eqnarray}
by Lemma \ref{lemma-2}, and \begin{equation}\label{pro-2}
 \sum_{n=1}^{\infty}(n-\lambda)|\Phi_n(\mathbf{F}; 0)|^2\leq \lambda,
 \end{equation}
 by  Lemma \ref{lemma-1} since $\mathbf{F}$ is also univalent analytic in $\mathbb{D}$. Thus (\ref{m-ine-1}) follows by (\ref{pro-1}) and (\ref{pro-2}) and the only if part of Theorem \ref{main-1} is proved.
 
 When $\lambda\in (0,1)$, (\ref{pro-2}) implies that  
\begin{equation}|\Phi_n(\mathbf{F}; 0)|\leq  \sqrt{\frac{\lambda}{1-\lambda}},  \nonumber
\end{equation}for all $\zeta\in \mathbb{D}, n\geq 1.$

When $\lambda\geq1$,  there is a positive integer $\mathbf{N}$ such that $\mathbf{N}\leq \lambda < \mathbf{N}+1$. Then (\ref{pro-2}) yields  that  
\begin{eqnarray}
|\Phi_n(\mathbf{F}; 0)| \leq\sqrt{\frac{\lambda+ \mathbf{E}(\lambda, \zeta)}{\mathbf{N}+1-\lambda}}:=\mathbf{E}_1(\lambda, \zeta), \nonumber
\end{eqnarray}
for all $\zeta\in \mathbb{D}, n\geq \mathbf{N}+1$.  Here $$ \mathbf{E}(\lambda, \zeta)=\sum_{k=1}^{N}(\lambda-k)|\Phi_k(\mathbf{F}; 0)|^2.$$
Meanwhile, it's obvious that 
$$|\Phi_n(\mathbf{F}; 0)|  \leq \sqrt{\sum_{k=1}^{\mathbf{N}}|\Phi_k(\mathbf{F}; 0)|^2}:=\mathbf{E}_2(\lambda, \zeta),$$ for all $n \in [1, \mathbf{N}].$ 

Then, from (\ref{pro-1}), we find that (\ref{m-ine-3}) is true for all $\zeta\in \mathbb{D}, n\geq 1$,  when we take 
$$C(\lambda, \zeta)=\sqrt{\frac{\lambda}{1-\lambda}},$$ for $\lambda\in (0,1)$, 
and  
$$C(\lambda, \zeta)=\mathbf{E}_1(\lambda, \zeta)+\mathbf{E}_2(\lambda, \zeta),$$
for $\lambda\geq 1$.  The only if part of Theorem \ref{main-3} is proved.

We proceed to prove the  if part.  We assume that (\ref{m-ine-3}) holds. If $f$ is not univalent analytic in $\mathbb{D}$, then there are two different points $\zeta, \alpha$ in $\mathbb{D}$ with $f(\zeta)=f(\alpha)$. We let
$$\mathcal{F}(z)=f(\sigma_{\zeta}(z))\,\, {\text{and}}\,\ \alpha=\sigma_{\zeta}(\beta).$$
Here $\sigma_{\zeta}$ is defined as in (\ref{sigma}). Then we get that 
\begin{equation}\label{proof}\mathcal{F}(0)=f(\zeta)=f(\alpha)=\mathcal{F}(\beta). \end{equation}
It is easy to see that $\beta \neq 0$ since $\zeta \neq \alpha$.

On the other hand, from (\ref{new}), we have 
\begin{equation}\label{proof-1}
\left[\frac{{\mathcal{F}}'(0)\beta}{\mathcal{F}(\beta)-\mathcal{F}(0)}\right]^{\lambda}=1+\sum_{n=1}^{\infty}\Phi_n(\mathcal{F}; 0){\beta}^n.
\end{equation}
Also, we see from the assumption and (\ref{pro-1}) that
\begin{equation}\label{proof-2}
|\Phi_{n}(\mathcal{F}; 0)|\leq C(\lambda, \zeta),
\end{equation} 
for all $n\geq 1.$
It follows from (\ref{proof-1}) and (\ref{proof-2}) that  
\begin{equation}
\frac{|{\mathcal{F}}'(0)\beta|^{\lambda}}{|\mathcal{F}(\beta)-\mathcal{F}(0)|^{\lambda}}\leq 1+\sum_{n=1}^{\infty}|\Phi_n(\mathcal{F}; 0)|{|\beta|}^n\leq 1+C(\lambda, \zeta)\frac{|\beta|}{1-|\beta|} <+\infty. \nonumber
\end{equation}
This contradicts (\ref{proof}) so that $f$ is univalent analytic in $\mathbb{D}$.  The if part of Theorem \ref{main-3} is proved and the proof of Theorem \ref{main-3} is finished.

Next, we assume that (\ref{m-ine-1}) holds, then we know from the above arguments that  (\ref{m-ine-3}) holds with $C(\lambda, \zeta)=\sqrt{\lambda/(1-\lambda)}$ when $\lambda\in (0,1)$, 
and  $C(\lambda, \zeta)=\mathbf{E}_1(\lambda, \zeta)+\mathbf{E}_2(\lambda, \zeta)$ when $\lambda\geq 1.$  By using the if part of Theorem \ref{main-3},  we see that $f$ is univalent analytic in $\mathbb{D}$. The proof of Theorem \ref{main-1} is complete. 
\end{proof}

\begin{proof}[Proof of Theorem \ref{main-2}]
First, noting that $\mathbf{F}$ is a full mapping if $f$ is a full mapping, we see from the Prawitz area theorem that (\ref{m-ine-2}) holds if $f$ is univalent analytic and a full mapping.
The only if part of Theorem \ref{main-2} is proved. 

We then assume (\ref{m-ine-2}) holds. By Theorem \ref{main-1} or \ref{main-3}, we get that $f$ is univalent analytic in $\mathbb{D}$. Taking $\zeta=0$ in (\ref{m-ine-2}), we see that 
\begin{equation}
\sum_{n=1}^{\infty}(n-\lambda)\left|\Phi_{n}(f; 0)\right|^2\equiv \lambda. \nonumber
\end{equation}
It follows from Lemma \ref{lemma-1} that $f$ is a full mapping.  The if part of Theorem \ref{main-2} is proved. 
\end{proof}
Now,  the proof of Theorem \ref{main-1}, \ref{main-2} and \ref{main-3} is done. 

\begin{remark}[on Theorem \ref{main-2}]
For the well-known Koebe function 
\[\kappa(z)=\frac{z}{(1-z)^2}, z\in \mathbb{D},\]
a simple computation yields that 
\[
\left[\frac{z}{\kappa(z)}\right]^\lambda = (1-z)^{2\lambda},\]
which implies that 
\[
\Phi_n(\kappa;0)=\binom{2\lambda}{n}(-1)^n.
\]
It follows from (\ref{pro-1}) and (\ref{m-ine-2}) with $\zeta=0$ that 
\[
\sum_{n=1}^{\infty}(n-\lambda)\left[\binom{2\lambda}{n}\right]^2=\lambda, \]
for all \(\lambda>0\). Here, the equality follows from the fact that $\kappa$ is a full mapping.
This means that our criterion reduces to an identity for all $\lambda>0$, while Aharonov's original criterion only covers the single case $\lambda=1$. This also reveals the underlying role of the Prawitz area theorem.
\end{remark}

\section{\bf{Remarks and the properties of (mixed) Aharonov sequence}}
 \begin{remark}
 We first remark that our results generalize some related ones of Aharonov in \cite{Ah}.  Let $f\in \mathcal{L}(\mathbb{D})$.  When $\lambda=1$, from (\ref{inv}) and (\ref{new}),  we have
 \begin{equation}
\frac{f'(z)(w-z)}{f(w)-f(z)}=1+\sum_{n=1}^{\infty}\Phi_n(f; z)(w-z)^{n}=1+\sum_{n=0}^{\infty}\phi_n(f; z)(w-z)^{n+1}. \nonumber
\end{equation}
It follows that $\Phi_n(f; z)=\phi_{n-1}(f; z)$ for $n\geq 1$.  Then we see from  \cite[Theorem 2]{Ah} that the inequality (\ref{a-1}) is equivalent to 
\begin{equation}
 \sum_{n=1}^{\infty} n|\phi_n(\mathbf{F}; 0)|^2=\sum_{n=2}^{\infty}(n-1)|\Phi_n(\mathbf{F}; 0)|^2  \leq 1. \nonumber
 \end{equation}
Where $\mathbf{F}$ is defined as in Lemma \ref{lemma-2}.

On the other hand,  when $\lambda=1$, we see from Lemma \ref{lemma-2} that the inequality (\ref{m-ine-1}) is equivalent to 
\begin{equation}
 \sum_{n=2}^{\infty}(n-1)|\Phi_n(\mathbf{F}; 0)|^2\leq 1. \nonumber
 \end{equation}
From these, we can see Theorem \ref{main-1} as a generalization of Theorem \ref{ah-1} if we assume that $f$ belongs to $ \mathcal{L}(\mathbb{D})$. By the same reason,  Theorem \ref{main-2} and \ref{main-3} also can be seen as some kind of extensions of Theorem \ref{ah-2} and \ref{ah-3}, respectively. 
\end{remark}

\begin{remark}
In 1982, Harmelin proved the following property for the Aharonov sequence, see  \cite[Theorem 1]{Ha-1}. 
\begin{proposition}
Let $f\in \mathcal{U}(\mathbb{D})$. Then we have 
\begin{equation}\label{harmelin-est}
(1-|z|^2)^{n+1}|\phi_n(f; z)|\leq \sum_{k=1}^{n}\binom{n-1}{k-1}\frac{|z|^{n-k}}{\sqrt{k}}, z\in \mathbb{D},\, n\geq 1. \nonumber
\end{equation}
\end{proposition}
We shall prove a similar result for the mixed Aharonov sequence $\{\Phi_n\}$.
\begin{proposition}\label{mixed-est}
Let $\lambda\in (0,1)$ and $f\in \mathcal{U}(\mathbb{D})$.  Then we have 
\begin{equation}\label{ggood}
(1-|z|^2)^n|\Phi_{n}(f; z)|\leq \sum_{j=0}^{n}\binom{\lambda}{n-j}\sum_{k=0}^{j}\binom{j-1}{j-k}\frac{\sqrt{\lambda}}{\sqrt{|k-\lambda|}}|{z}|^{n-k}, z\in \mathbb{D},\, n\geq 1. 
\end{equation}
\end{proposition}
\begin{proof}
For fixed $\zeta \in \mathbb{D}$. We set
$$\mathbf{F}_1(w)=f\left(\frac{w-\zeta}{1-\bar{\zeta}w}\right),$$
and $$w=\sigma_\zeta(z)=\frac{z+\zeta}{1+\bar{\zeta}z}. $$
Then we see that $f(z)={\bf F}_1 \circ \sigma_\zeta(z).$ Consequently, by Lemma \ref{lemma-2}, we get that, for $n\geq 0$, 
\begin{eqnarray}\label{good}
\lefteqn{\Phi_{n}(f; z)=\Phi_{n}(\mathbf{F}_{1}\circ \sigma_\zeta; z) } \\
&&=\sum_{j=0}^{n}(-1)^{n-j}\binom{\lambda}{n-j}\sum_{k=0}^{j}\binom{j-1}{j-k}\frac{(-\bar{\zeta})^{n-k}(1-|\zeta|^2)^{k}}{(1+\bar{\zeta}z)^{n+k}}\Phi_{k}(\mathbf{F}_{1}; w). \nonumber
\end{eqnarray}
Take $\zeta=-z$ in (\ref{good}), we obtain that, for $n\geq 0$, 
\begin{equation}\label{good-1}
\Phi_{n}(f; z)=\sum_{j=0}^{n}(-1)^{n-j}\binom{\lambda}{n-j}\sum_{k=0}^{j}\binom{j-1}{j-k}(\bar{z})^{n-k}(1-|z|^2)^{-n}\Phi_{k}(\mathbf{F}_{1}; 0). 
\end{equation}

On the other hand, from (\ref{pro-2}), we know that, for any $n\geq 1$,
$\Phi_{n}(\mathbf{F}_{1}; 0)\leq \frac{\sqrt{\lambda}}{\sqrt{n-\lambda}}.$
Also, noting that $\Phi_{0}(\mathbf{F}_{1}; 0)=1$, so we can write, for any $n\geq 0,$ 
 \begin{equation}\label{good-2}\Phi_{n}(\mathbf{F}_{1}; 0)\leq \frac{\sqrt{\lambda}}{\sqrt{|n-\lambda|}}.\end{equation}
Then (\ref{ggood}) follows by (\ref{good-2}) and (\ref{good-1}). The proposition is proved.
\end{proof}
\end{remark}
\begin{remark}
For $\lambda\in(0,1)$, the estimate in Proposition \ref{mixed-est} reveals that the growth rate of $|\Phi_n(f;z)|$ as $|z|\to 1^-$ is at most $(1-|z|^2)^{-n}$, and the explicit bound is universal, depending only on $n$, $|z|$, and $\lambda$, not on the specific univalent function $f$.
\end{remark}
 Finally, we consider the univalent analytic functions with a quasiconformal extension to the whole complex plane.  We refer the reader to \cite{LV} and \cite{L} for the introduction to the theory of quasiconformal mappings.  
 
Let $f$ be a univalent analytic function in $\mathbb{D}$, which can be  extended to a quasiconformal mapping(still denoted by $f$) in $\widehat{\mathbb{C}}$. We denote by $\mu_f(z)$ the complex dilatation of $f$.  Let $\mathbf{r}$ be a Riemann mapping from $\mathbb{D}_e$ to $\widehat{\mathbb{C}}-\overline{ f(\mathbb{D})}$. We see that $\mathbf{r}$ can be extended to a homeomorphism(still denoted by $\mathbf{r}$) from $\widehat{\mathbb{C}}-\mathbb{D}$ to $\widehat{\mathbb{C}}-f(\mathbb{D})$. Then $h=f^{-1}\circ \mathbf{r}|_{\mathbb{T}}$ is a quasisymmetric homeomorphism from $\mathbb{T}$ to itself.  It follows that $\bar{f}=\mathbf{r}^{-1}\circ f|_{\mathbb{D}_e}$ is a quasiconformal extension of $h^{-1}$
to $\mathbb{D}_e$ and $\bar{f}$ has the same complex dilatation as $f$ in $\mathbb{D}_e$.  We let $\tilde{f}=\rho \circ \bar{f} \circ \rho|_{\mathbb{D}}$, here $\rho(z)={\bar{z}}^{-1}$.
Then we see that $\tilde{f}$ is a quasiconformal extension of $h^{-1}$ to $\mathbb{D}$. Denote by $\mu_{\tilde{f}}(z)$ the complex dilatation of $\tilde{f}$, we get that 
$|\mu_{\tilde{f}}(z)|=|\mu_f({\bar{z}}^{-1})|, \, z\in \mathbb{D}$. Therefore, by Corollary 2.2 in \cite{SW}, we obtain that
\begin{equation}\label{ine-l}
U_f^2(z)\leq \frac{1}{\pi}\iint_{\mathbb{D}} \frac{|\mu_{f}({\bar{w}}^{-1})|^2}{1-|\mu_f({\bar{w}}^{-1})|^2}\frac{dudv}{|1-\bar{z}w|^4}.
\end{equation}

On the other hand, for fixed $z\in \mathbb{D}$, we know that the Koebe transformation $\mathbf{K}_{f}(z; \zeta)$ of ${f} $,  defined as in (\ref{Koe}), belongs to $S$. Then we see that $\mathbf{G}_{f}(z; \zeta):=\frac{1}{\mathbf{K}_{f}(z; \zeta^{-1})} \in \Sigma$ and set 
\begin{equation}\label{last-1}
\mathbf{G}_{f}(z; \zeta)=\frac{(1-|z|^2){f}'(z)}{f\left(\sigma_{z}(\zeta^{-1})\right)-{f}(z)} =\zeta+\sum_{n=0}^{\infty}\Psi_n({f}; z)\zeta^{-n}, \, \zeta \in \mathbb{D}_e-\{\infty\}.
\end{equation}

Then, by checking the arguments of the derivation of Prawitz's inequality (\ref{main-ine}) presented in the Introduction, we get that 
\begin{equation}\label{eqn}
\sum_{n=1}^{\infty}n|\Psi_n(f; z)|^2=(1-|z|^2)^{2}U_{f}^2(z).
\end{equation} 
Combining (\ref{ine-l}), (\ref{eqn}), we obtain that
\begin{eqnarray}\label{eqn-1}
\sum_{n=1}^{\infty}n|\Psi_n(f; z)|^2&=&(1-|z|^2)^{2}U_f^2(z) \nonumber \\&\leq& \frac{(1-|z|^2)^{2}}{\pi}\iint_{\mathbb{D}} \frac{|\mu_{f}({\bar{w}}^{-1})|^2}{1-|\mu_f({\bar{w}}^{-1})|^2}\frac{dudv}{|1-\bar{z}w|^4}.
\end{eqnarray}
 
Also, we see from (\ref{inv}) that
\begin{eqnarray}\label{last-2}
\frac{f'(z)}{f\left(\sigma_z({\zeta}^{-1})\right)-f(z)}&=& \frac{1}{\frac{1+z\zeta}{\zeta+\bar{z}}-z}+\sum_{n=0}^{\infty} \phi_{n}(f ; z)(\frac{1+z\zeta}{\zeta+\bar{z}}{-z)}^n
\\
&=&   \frac{\zeta+\bar{z}}{1-|z|^2}+\sum_{n=0}^{\infty} \phi_{n}(f ; z)(1-|z|^2)^n \zeta ^{-n}(1+\frac{\bar{z}}{\zeta})^{-n}.  \nonumber
\end{eqnarray}

Note that 
\begin{equation}\label{last-3}
 (1+\frac{\bar{z}}{\zeta})^{-n}=\sum_{k=0}^{\infty}\binom{-n}{k} {\bar{z}}^{k}{\zeta}^{-k}=\sum_{k=0}^{\infty}\binom{n+k-1}{k} {(-\bar{z})}^{k}{\zeta}^{-k}.\end{equation}
It follows from (\ref{last-2}) and (\ref{last-3})  that
\begin{eqnarray}
\mathbf{G}_{f}(z; \zeta)&=&\zeta+\bar{z}+\sum_{n=0}^{\infty}\left[\phi_{n}(f; z)(1-|z|^2)^{n+1} \zeta ^{-n}\right]\left[ \sum_{k=0}^{\infty}\binom{n+k-1}{k}{(-\bar{z})}^{k}{\zeta}^{-k} \right] \nonumber \\
&=&\zeta+\bar{z}+(1-|z|^2)\phi_0(f ; z)\nonumber \\
& & \quad+\sum_{n=1}^{\infty}\left[\phi_{n}(f; z)(1-|z|^2)^{n+1} \zeta ^{-n}\right]\left[\sum_{k=0}^{\infty}\binom{n+k-1}{k}{(-\bar{z})}^{k}{\zeta}^{-k} \right]. \nonumber
\end{eqnarray}

Then we get from (\ref{last-1}) that
\begin{equation}
\Psi_0(f; z)=\bar{z}+(1-|z|^2)\phi_0(f; z)=\bar{z}-\frac{1}{2}(1-|z|^2)N_{f}(z); \nonumber
\end{equation}
\begin{equation}\label{last-4}
\Psi_n(f; z)=\sum_{k=1}^{n}\binom{n-1}{n-k}(-\bar{z})^{n-k}(1-|z|^2)^{k+1} \phi_k(f; z),\, n\geq 1. 
\end{equation}
In particular, $$\Psi_1(f; z)=(1-|z|^2)^{2}\phi_1(f; z)=-\frac{1}{6}(1-|z|^2)^{2}S_{f}(z).$$
Therefore we know that any $\Psi_n(n\geq 1)$ is Möbius invariant like $\phi_n(n\geq 1)$.

\begin{remark}
The following result was obtained in \cite[Lemma 5]{Ha-1} by Harmelin.
\begin{proposition}
Let $f$ be a univalent function in $\mathbb{D}$, which can be extended to a quasiconformal mapping(still denoted by $f$) in $\widehat{\mathbb{C}}$. Let $\|\mu_f\|_{\infty}$ be the norm of the complex dilatation $\mu_f(z)$ of $f$. Then, for all $\zeta\in \mathbb{D}$, it holds that 
\begin{equation}\label{harmelin}
\sum_{n=1}^{\infty}n\left|\sum_{k=1}^n \binom{n-1}{k-1} (-\bar{\zeta})^{n-k}(1-|\zeta|^2)^{k}\phi_k(f; \zeta)\right|^2\leq \|\mu_f\|_{\infty}^2.
\end{equation}
\end{proposition}
By using (\ref{eqn-1}) and (\ref{last-4}), we can get the following inequality similar to (\ref{harmelin}).
\begin{eqnarray}\label{jia}
\lefteqn{\sum_{n=1}^{\infty}n\left|\sum_{k=1}^n \binom{n-1}{k-1} (-\bar{\zeta})^{n-k}(1-|\zeta|^2)^{k}\phi_k(f; \zeta)\right|^2}  \\
 &&\quad\quad\quad\quad\quad \quad\quad \quad\quad  \leq
\frac{1}{\pi}\iint_{\mathbb{D}} \frac{|\mu_{f}({\bar{w}}^{-1})|^2}{1-|\mu_f({\bar{w}}^{-1})|^2}\frac{dudv}{|1-\bar{\zeta}w|^4},\, \zeta\in \mathbb{D}.\nonumber
\end{eqnarray}
Let $k=\|\mu_f\|_{\infty}$. Then it follows from (\ref{jia}) that 
\begin{eqnarray} 
\lefteqn{\sum_{n=1}^{\infty}n\left|\sum_{k=1}^n \binom{n-1}{k-1} (-\bar{\zeta})^{n-k}(1-|\zeta|^2)^{k}\phi_k(f; \zeta)\right|^2} \nonumber \\
 &&\quad\quad\quad \leq \frac{k^2}{1-k^2}\frac{1}{\pi}\iint_{\mathbb{D}}\frac{dudv}{|1-\bar{\zeta}w|^4}=\frac{k^2}{1-k^2}(1-|\zeta|^2)^{-2},\, \zeta\in \mathbb{D}.\nonumber
\end{eqnarray}
Here the last equality follows from the well-known integral
\[
\frac{1}{\pi}\iint_{\mathbb{D}}\frac{dudv}{|1-\bar{\zeta}w|^4} = \frac{1}{(1-|\zeta|^2)^2},\,\, |\zeta|<1.
\]
\end{remark}

\begin{remark}
From (\ref{eqn-1}),  we see that, for any $n \geq 1$, it holds that
\begin{equation}\label{ds}
\frac{n|\Psi_n(f; z)|^2}{(1-|z|^2)^2}\leq U_f^2(z) \leq \frac{1}{\pi}\iint_{\mathbb{D}} \frac{|\mu_{f}({\bar{w}}^{-1})|^2}{1-|\mu_f({\bar{w}}^{-1})|^2}\frac{dudv}{|1-\bar{z}w|^4}.
\end{equation}
It is known that we can use the Schwarzian derivative $S_f(z)=-\frac{6\Psi_1(f; z)}{(1-|z|^2)^2}$ and the quantity $U_f(z)$ to describe some subclasses of the universal Teichm\"uller space. For example, the little Teichm\"uller space,  see \cite{GS}, \cite{Sh-1}, \cite{Sh-2};  Weil-Petersson Teichm\"uller space, see \cite{B-1}, \cite{Sh-2}, \cite{Sh-3}, \cite{TT}; BMO(VMO)-Teichm\"uller space, see \cite{AZ}, \cite{SW}. It is natural to study 
\begin{question}\label{prob-1}
Characterize the subclasses of the universal Teichm\"uller space in terms of $\Psi_n\,(n\geq 2)$.
\end{question}
\end{remark}
In view of the inequality (\ref{ds}), it seems that one direction of the Problem \ref{prob-1} is easy, but another one is not.  Also, it is interesting to consider the following
\begin{question}
 Find Prawitz's inequality, analogue to the inequality (\ref{ine-l}), for the univalent analytic functions in $\mathbb{D}$ with a quasiconformal extension to the whole complex plane.
\end{question}
\begin{remark}
If we can obtain an analogue inequality to (\ref{ine-l}), then we may use it to establish some new results about the theory of integral means spectrum. 
\end{remark}
\section{{\bf Acknowledgements}}
The author was supported by National Natural Science Foundation of China (Grant No. 11501157). 

\begin{spacing}{1.2}

\end{spacing}


\begin{thebibliography}{99}
\bibitem{Ah}Aharonov D., {\em A necessary and sufficient condition for univalence of a meromorphic function}, Duke Math. J., 36(1969), pp. 599-604.

\bibitem{AG}Astala K., Gehring F.,  {\em Injectivity, the BMO norm and the Teichm\"uller space}, J. Anal. Math.,  46(1986), pp. 16-57.

\bibitem{AZ}Astala K., Zinsmeister M.,  {\em Teichm\"uller spaces and BMOA}, Math. Ann., 289(1991), pp. 613-625.


\bibitem{Ba}Bazilevic I.,  {\em On a criterion of univalence of regular functions and the disposition of their coefficients}, Math. USSR Sb., 3 (1967), pp. 123-137.

\bibitem{B-1}Bishop C., {\em Function theoretic characterizations of Weil-Petersson curves}, Rev. Mat. Iberoam., 38 (2022), pp. 2355-2384.










\bibitem{GS}Gardiner F., Sullivan D., {\em Symmetric structures on a closed curve}, Amer. J. Math., 114(1992), pp. 683-736.



\bibitem{Ha}Harmelin R., {\em Bergman kernel functions and univalent criteria}, J. Anal. Math.,  41 (1982), pp. 249-258.

\bibitem{Ha-1}Harmelin R., {\em Aharonov invariants and univalent functions},  Israel J. Math., 43 (1982), pp.  244-254.



\bibitem{HS-1}Hedenmalm H., Shimorin S., {\em Weighted Bergman spaces and the integral means spectrum of conformal mappings},  Duke Math. J. 127 (2005), pp. 341-393.



\bibitem{L}Lehto O., {\em Univalent functions and Teichm\"uller spaces}, Springer-Verlag, 1987.

\bibitem{LV}Lehto O., Virtanen K., {\em Quasiconformal mappings in the plane}, Second edition, Springer-Verlag, New York-Heidelberg, 1973.


\bibitem{Mi}Milin I., {\em Univalent Functions and Orthonormal Systems}, Trans. Math. Monogr. Vol. 49, Amer. Math. Soc., Providence, 1977.

\bibitem{Po}Pommerenke C., {\em Univalent functions}, Vandenhoeck and Ruprecht, G\"ottingen, 1975.


\bibitem{Pra}Prawitz H., {\em \"Uber Mittelwerte analytischer Funktionen}, Arkiv Mat. Astronomy Fysik., 20A (1927–1928), pp. 1-12.

\bibitem{Sh-1}Shen Y., {\em On Grunsky operator}, Sci. China Ser.  A,  50 (2007), no. 12, pp. 1805-1817.

\bibitem{Sh-2}Shen Y., {\em Faber polynomials with applications to univalent functions with quasiconformal extensions}, Sci. China
Ser. A,  52 (2009), pp. 2121-2131.

\bibitem{Sh-3}Shen Y., {\em Weil-Petersson Teichmüller space},  Amer. J. Math., 140 (2018),  pp. 1041-1074.

\bibitem{SW}Shen Y.,  Wei H., {\em Universal Teichmüller space and BMO},  Adv. Math.,  234 (2013), pp. 129-148

\bibitem{Su}Sugawa T., {\em Aharonov invariants revisited}, Anal. Math. Phys., 10(2020), no. 3, Paper No. 29, 12 pp.


\bibitem{TT}Takhtajan L., Teo T., {\em Weil–Petersson metric on the universal Teichm\"uler space}, Mem. Amer. Math. Soc., 183 (861), 2006.


\end{thebibliography}
\end{document}